\documentclass[11pt]{amsart}

\usepackage{amsfonts}
\usepackage{amssymb}
\usepackage{amsthm}
\usepackage{amsmath}
\usepackage{hyperref}
\usepackage{enumitem}
\usepackage{tikz}
\usepackage{pgf}
\usepackage{graphicx}
\usepackage[top=3.5cm, bottom=3.5cm, left=2.5cm, right=2.5cm]{geometry}
\usepackage[capitalise]{cleveref}

\newtheorem{prop}{Proposition}[section]
\newtheorem{theorem}[prop]{Theorem}
\newtheorem{lemma}[prop]{Lemma}

\newtheorem*{theorem*}{Theorem}

\crefname{prop}{Proposition}{Propositions}

\theoremstyle{definition}

\theoremstyle{remark}

\newtheorem{examples}[prop]{Examples}
\newtheorem*{remark*}{Remark}
\newtheorem{remark}[prop]{Remark}

\newtheorem{que}[prop]{Question}

\theoremstyle{theorem}

\newcommand{\R}{\mathbb{R}}

\newcommand{\N}{\mathbb{N}}

\newcommand{\Z}{\mathbb{Z}}
\newcommand{\Q}{\mathbb{Q}}

\newcommand{\eps}{\varepsilon}

\newcommand{\h}{\text{\textup{h}}}

\newcommand{\Aut}{\text{\textup{Aut}}\,}

\newcommand{\Aff}{\mathcal{A}}

\newcommand*{\Fol}{\mathop{\mathrm{F\o l}}\nolimits}

\renewcommand{\ge}{\geqslant}
\renewcommand{\le}{\leqslant}
\renewcommand{\geq}{\geqslant}
\renewcommand{\leq}{\leqslant}

\numberwithin{equation}{section}

\title{Geometric amenability in totally disconnected locally compact groups}
\date{}
\author{Romain Tessera}
\address{Institut de Math\'ematiques de Jussieu-Paris Rive Gauche, France}
\email{romain.tessera@imj-prg.fr}
\author{Matthew Tointon}
\address{School of Mathematics, University of Bristol, United Kingdom}
\email{m.tointon@bristol.ac.uk}
\thanks{For part of this project M. Tointon was supported by the Stokes Research Fellowship from Pembroke College, Cambridge}

\makeatletter
\@namedef{subjclassname@2020}{\textup{2020} Mathematics Subject Classification}
\makeatother

\begin{document}
\maketitle
\begin{abstract}
We give a short geometric proof of a result of Soardi \& Woess and Salvatori that a quasitransitive graph is amenable if and only if its automorphism group is amenable and unimodular. We also strengthen one direction of that result by showing that if a compactly generated totally disconnected locally compact group admits a proper Lipschitz action on a bounded-degree amenable graph then that group is amenable and unimodular. We pass via the notion of \emph{geometric amenability} of a locally compact group, which has previously been studied by the second author and is defined by analogy with amenability, only using right F\o lner sets instead of left F\o lner sets. We also introduce a notion of \emph{uniform geometric non-amenability} of a locally compact group, and relate this notion in various ways to actions of that group on graphs and to its modular homomorphism.
\end{abstract}

\setcounter{tocdepth}{1}
\tableofcontents

\section{Introduction}

A well-known result of Soardi \& Woess \cite[Corollary 1]{soardi-woess} states that a vertex-transitive graph is amenable if and only if its automorphism group is amenable and unimodular. Salvatori \cite[Theorem 1]{salvatori} generalised this result to quasitransitive graphs. Benjamini, Lyons, Peres and Schramm \cite[Remarks 3.11 \& 6.3]{blps} later gave an alternative proof. Each of these proofs emerged as a corollary of broader work: Soardi and Woess's and Salvatori's proofs came out of work on random walks (see also \cite{sc-woess}), whilst one direction of Benjamini, Lyons, Peres and Schramm's proof came out of work on percolation. Lyons and Peres subsequently gave a geometric proof of this direction \cite[Proposition 8.14]{ly-per}. The initial purpose of this paper is to offer a quick, geometric proof of the Soardi--Woess--Salvatori theorem, and strengthen one direction of it by showing that the presence of a proper action on a bounded-degree amenable graph that is merely Lipschitz, and not necessarily quasitransitive, is enough to imply that a compactly generated totally disconnected locally compact group is amenable and unimodular.

We start by presenting the necessary definitions. Let $\Gamma$ be a graph. Given a set $A\subseteq\Gamma$ of vertices, we define the \emph{boundary} $\partial A$ to consist of those vertices that do not belong to $A$ but have a neighbour in $A$. The \emph{(vertex) Cheeger constant} $\h(\Gamma)$ of $\Gamma$ is defined via
\[
\h(\Gamma)=\inf_{\substack{A\subseteq V\\|A|<\infty}}\frac{|\partial A|}{|A|},
\]
and $\Gamma$ is then called \emph{amenable} if $\h(\Gamma)=0$.

The group $\Aut(\Gamma)$ of automorphisms of $\Gamma$ is a locally compact group with respect to the topology of pointwise convergence, which is metrisable. Every closed subgroup of $\Aut(\Gamma)$ is then a compactly generated, totally disconnected, locally compact group in which vertex stabilisers are compact and open. This is classical, and treated in detail in \cite{tt.trof,trof,woess} for instance.

Now let $G$ be an arbitrary compactly generated locally compact group. An action of $G$ by automorphisms on $\Gamma$ is called \emph{continuous} if the homomorphism $G\to\Aut(\Gamma)$ it induces is continuous. \textbf{In this paper, we assume by definition that all actions of topological groups on graphs are continuous}. An action is called \emph{proper} if vertex stabilisers are compact in $G$; in particular, every closed subgroup of $\Aut(\Gamma)$ acts properly on $\Gamma$. It is called \emph{transitive} if there is a unique orbit of vertices, and \emph{quasitransitive} if there are only finitely many orbits of vertices. A theorem of Abels \cite{ab} (see \cite[Proposition 2.E.9]{CH} or \cite[Theorem 2.2\textsuperscript{+}]{KM}) states that every compactly generated, totally disconnected, locally compact group admits a transitive proper continuous action on some connected, locally finite graph.

Suppose $G$ is a compactly generated locally compact group acting properly continuously by automorphisms on a locally finite graph $\Gamma$. Given $C>0$, we say that this action is \emph{$C$-Lipschitz} with respect to a given compact symmetric generating set $S$ if in each orbit there exists some vertex $x$ such that the orbit map
 \[
\begin{array}{ccc}
(G,S)&\to&G\cdot x\\
g&\mapsto &g\cdot x
\end{array}
\]
is $C$-Lipschitz (here, $(G,S)$ means the group $G$ endowed with the word metric with respect to $S$). We will say that the action is \emph{$C$-Lipschitz} to mean that there exists some $S$ with respect to which it is $C$-Lipschitz, and simply \emph{Lipschitz} to mean that there exists some $C>0$ such that the action is $C$-Lipschitz.

Every quasitransitive action is Lipschitz. Indeed, if vertices $x_1,\ldots,x_k$ are representatives of the orbits, then the action is $C$-Lipschitz where $C$ is the maximal integer $n$ such that $S\cdot x_i\subseteq B(x_i,n)$ for every $i=1,\ldots,k$.

A locally compact group $G$ admits a \emph{(left) Haar measure} $\mu$, the properties of which include that
\begin{enumerate}[label=(\roman*)]
\item$\mu(K)<\infty$ if $K$ is compact,
\item$\mu(U)>0$ if $U$ is open and nonempty,
\item$\mu(gA)=\mu(A)$ for every Borel set $A\subseteq G$ and every $g\in G$, and
\item\label{item:Haar.unique} if $\mu'$ is another Haar measure on $G$ then there exists $\lambda>0$ such that $\mu'=\lambda\cdot\mu$.
\end{enumerate}
See \cite[\S15]{hew-ross} for a detailed introduction to Haar measures. Note that since a right translate of a Haar measure is again a Haar measure, by property \ref{item:Haar.unique} there exists a homomorphism $\Delta_G:G\to\R^+$, called the \emph{modular homomorphism}, such that
\[
\mu(Ag)=\Delta_G(g)\mu(A)
\]
for every Borel set $A$. Property \ref{item:Haar.unique} also implies that $\Delta_G$ depends only on $G$, and not on $\mu$. The group $G$ is called \emph{unimodular} if $\Delta_G\equiv1$, in which case $\mu(gA)=\mu(Ag)=\mu(A)$ for every Borel set $A\subseteq G$ and every $g\in G$. Note that if $G$ is unimodular then we may define another Haar measure $\mu'$ by $\mu'(A)=\mu(A^{-1})$, and then by property \ref{item:Haar.unique} we have $\mu=\mu'$ so that $\mu$ is symmetric.

A locally compact group $G$ with Haar measure $\mu$ is called \emph{amenable} if for every compact subset $K\subseteq G$ and every $\eps>0$ there exists a compact set $U\subseteq G$ of positive measure such that
\[
\frac{\mu(KU)}{\mu(U)}\le1+\eps.
\]
The group $G$ is called \emph{geometrically amenable} if for every compact subset $K\subseteq G$ and every $\eps>0$ there exists a compact set $U\subseteq G$ of positive measure such that
\[
\frac{\mu(UK)}{\mu(U)}\le1+\eps.
\]
In a unimodular group these notions coincide by the symmetry of $\mu$. In a non-unimodular group $G$, there exists $k\in G$ such that $\Delta_G(k)>1$, and then since $\mu(Uk)=\Delta_G(k)\mu(U)$ for all compact sets $U$ of positive measure, $G$ is not geometrically amenable. Thus, we have the following lemma, previously noted by the first author \cite[\S11]{tessera.sobolev}.
\begin{lemma}\label{lem:amen-unimod}
A locally compact group is geometrically amenable if and only if it is amenable and unimodular.
\end{lemma}
This opens a new avenue to understanding which groups are both amenable and unimodular, which we exploit to prove the following result, which in some sense shows that simultaneous amenability and unimodularity of an \emph{arbitrary} compactly generated, totally disconnected, locally compact group necessarily reflects an action of that group on an amenable graph.
\begin{theorem}\label{thm:amen+unimod}
Let $G$ be a compactly generated, totally disconnected, locally compact group. Then the following are equivalent:
\begin{enumerate}
\item $G$ is amenable and unimodular;\label{item:amen+unimod}
\item $G$ is geometrically amenable;\label{item:geom.amen}
\item there exists a bounded-degree amenable graph admitting a proper Lipschitz action of $G$;\label{item:exists.Lip}
\item there exists an amenable graph admitting a proper quasitransitive action of $G$;\label{item:exists.quasi}
\item there exists an amenable graph admitting a proper transitive action of $G$;\label{item:exists.VT}
\item every graph admitting a proper quasitransitive action of $G$ is amenable.\label{item:all.quasi}
\end{enumerate}
\end{theorem}
The implication \eqref{item:exists.VT} $\implies$ \eqref{item:exists.quasi} is trivial, and we discussed the implication \eqref{item:exists.quasi} $\implies$ \eqref{item:exists.Lip} above. The implication \eqref{item:all.quasi} $\implies$ \eqref{item:exists.VT} follows from Abels's theorem, whilst of course the equivalence \eqref{item:amen+unimod} $\iff$ \eqref{item:geom.amen} is a special case of \cref{lem:amen-unimod}.

The equivalence of \eqref{item:amen+unimod}, \eqref{item:exists.quasi} and \eqref{item:all.quasi} recovers the Soardi--Woess--Salvatori theorem. However, we prove the implication \eqref{item:amen+unimod} $\implies$ \eqref{item:all.quasi} via the following version of the Soardi--Woess--Salvatori theorem, in particular giving a more direct proof of that result than any previous reference we are aware of.
\begin{theorem}\label{thm:sw}
Suppose $\Gamma$ is a connected, locally finite graph, and that $G$ is locally compact group admitting a proper quasitransitive action on $\Gamma$. Then $\Gamma$ is amenable if and only if $G$ is amenable and unimodular.
\end{theorem}
We prove the one outstanding implication of \cref{thm:amen+unimod}, \eqref{item:exists.Lip} $\implies$ \eqref{item:geom.amen}, via the following result, which we believe to be completely new.
\begin{theorem}\label{thm:LipActionOnAmenable}
Suppose $G$ is a compactly generated, totally disconnected, locally compact group admitting a proper Lipschitz action on a bounded-degree amenable graph. Then $G$ is geometrically amenable.
\end{theorem}
 If one is willing to replace Lipschitz by $1$-Lipschitz, then we can drop the assumption that the graph must have bounded degree (see the first part of \cref{thm:geo.amen<->action.simple}).
Note that, thanks to the absence of any quasitransitivity assumption, \cref{thm:sw} is a significant strengthening of one direction of the Soardi--Woess--Salvatori theorem.

We prove \cref{thm:sw} in \cref{sec:sw}. \cref{thm:LipActionOnAmenable} follows from the first part of \cref{thm:geo.amen<->action.simple}, which itself results from \cref{prop:LipActionOnAmenable}.

\begin{remark}\label{rem:properness}
The properness of the actions in statements \eqref{item:exists.Lip}--\eqref{item:all.quasi} of \cref{thm:amen+unimod} cannot be removed. On the one hand, if $\Gamma$ is Cayley graph of degree $d$ on some group $H$, and $G$ is a free group of rank greater than $d$, then one may define a transitive (and hence quasitransitive and Lipschitz) action of the free group $G$ on $\Gamma$ by projecting $F_r\to H$ and letting $H$ act on $\Gamma$ by translations. In particular, amenability of $\Gamma$ in this instance does not imply amenability of $G$. Conversely, defining them as HNN-extensions, one can let the lamplighter and solvable Baumslag--Solitar groups act faithfully and transitively on regular trees (see for instance \cite{meier}), so that amenability and unimodularity of $G$ does not preclude the existence of improper transitive actions of $G$ on nonamenable graphs.
\end{remark}

\begin{remark}
Statement \eqref{item:all.quasi} of \cref{thm:amen+unimod} cannot be strengthened to say that every bounded-degree graph admitting a proper \emph{Lipschitz} action of $G$ is amenable.
For example, if $G$ is a finitely generated group and $T$ is a tree then the obvious action of $G$ on $G\times T$ is both proper and Lipschitz.
\end{remark}

In light of the lamplighter and Baumslag--Solitar examples described in \cref{rem:properness}, one might reasonably wonder whether Sol admits a faithful, transitive action on a regular tree. In \cref{prop:Sol} we show that it does not, in fact, admit any quasitransitive action on any non-amenable locally finite graph.

\subsection*{Uniform non-amenability}
Arzhantseva, Burillo, Lustig, Reeves, Short and Ventura \cite{ablrsv} define a uniform notion of non-amenability for finitely generated groups (Osin \cite{osin.weak.amen} considers a related notion called \emph{weak amenabilty}). In this paper we extend this definition to locally compact groups, and to geometric amenability. First, given a compactly generated locally compact group $G$, we follow Arzhantseva et.\ al.\ in defining
\[
\Fol G=\inf_S\inf_U\frac{\mu(SU\setminus U)}{\mu(U)},
\]
where the infima are over all compact symmetric generating sets $S$ for $G$ and all compact subsets $U\subseteq G$. (In fact, this differs slightly from Arzhantseva et.\ al.'s definition in that they consider the interior boundary, where we consider the exterior boundary.) If $G$ is amenable then $\Fol G=0$ by definition; we call a group $G$ satisfying $\Fol G>0$ \emph{uniformly non-amenable}. 

%\begin{remark}We briefly pause here to observe that $\Fol G$ does not depend on a choice of Haar measure. This is related to our choice of a {\it vertex} boundary in the Cayley graph $(G,S)$. Indeed the set of edges $G\times S$ is naturally equipped with the $G$-invariant measure $\mu\times\mu$. But then, multiplying the Haar measure by a scalar $t$ multiplies the measure of the edge-boundary by $t^2$, which in turn would multiply the corresponding version of $\Fol G$ by $t$. This explains why we decided from the beginning to consider vertex boundaries instead of edges boundary in our definition of amenability. Indeed, This distinction is of course only relevant when the graphs have unbounded degree.  \end{remark}

In the context of the present work it is natural to define analogously uniform \emph{geometric} non-amenability. Given a compactly generated locally compact group $G$, we therefore set
\[
\Fol^\ast G=\inf_S\inf_U\frac{\mu(US\setminus U)}{\mu(U)},
\]
where again the infima are over all compact symmetric generating sets $S$ for $G$ and all compact subsets $U\subseteq G$. If $G$ is geometrically amenable then $\Fol^\ast G=0$ by definition, and we call a group $G$ \emph{uniformly geometrically non-amenable} if $\Fol^\ast G>0$.

Recall that a connected Lie group is generated by any neighbourhood of the identity. It follows that if $G$ is such a group then $\Fol G=0$ and $\Fol^\ast G=0$; indeed, if $U$ is any compact subset of positive measure, and $(S_n)_{n=1}^\infty$ is a sequence of compact symmetric neighbourhoods of the identity converging to the identity, then
\[
\frac{\mu(S_nU\setminus U)}{\mu(U)}\to0,\qquad\qquad\frac{\mu(US_n\setminus U)}{\mu(U)}\to0.
\]
These notions are therefore more appropriately studied in the setting of totally disconnected locally compact groups.

It turns out that, in that setting, these notions relate to unimodularity and the presence of certain actions on graphs in a number of ways that are strongly analogous to \cref{thm:amen+unimod}. For example,
the following statement (which we prove in a more detailed form in \cref{prop:deltaDenseAlmostAm}) shows that in an amenable group $G$, uniform geometric non-amenability can be characterised in terms of the modular homomorphism, just as geometric amenability can be by the equivalence \eqref{item:amen+unimod} $\iff$ \eqref{item:geom.amen}.
\begin{theorem}\label{thm:deltaDense.simple}
Suppose $G$ is a compactly generated totally disconnected locally compact group. Suppose further that $G$ is amenable and non-unimodular. Then $\Fol^\ast G=0$ if and only if the image of $G$ under the modular homomorphism is dense in $\R^*_+$.
\end{theorem}
The following result, on the other hand, is directly analogous to the equivalence \eqref{item:geom.amen} $\iff$ \eqref{item:exists.VT}.
\begin{theorem}\label{prop:almost.amen.trans.seq}
Let $G$ be a compactly generated locally compact totally disconnected group. Then $\Fol^\ast G=0$ if and only if there exists a sequence of $G$-transitive proper locally finite graphs $\Gamma_n$ which are asymptotically amenable in the sense that there exists a sequence $A_n\subseteq\Gamma_n$ such that $|\partial A_n|/|A_n|\to 0$.
\end{theorem}
\cref{prop:almost.amen.trans.seq} actually follows from the more refined \cref{prop:Folstar.versus.h}, below.

Finally, we have an analogue of the equivalence \eqref{item:geom.amen} $\iff$ \eqref{item:exists.Lip}. To state it requires a further definition. Suppose $G$ is a compactly generated locally compact group acting properly on a locally finite graph $\Gamma$. Given $C>0$, we say that this action is \emph{contingently $C$-Lipschitz} if for every $x\in\Gamma$ there exists a compact symmetric generating subset $S_x\subset G$ such that the orbit map
 \[
\begin{array}{ccc}
(G,S_x)&\to&G\cdot x\\
g&\mapsto &g\cdot x
\end{array}
\]
is $C$-Lipschitz. We will say that the action is \emph{contingently Lipschitz} to mean that there exists some $C>0$ such that the action is contingently $C$-Lipschitz.

\begin{theorem}\label{thm:geo.amen<->action.simple}
Suppose $G$ is a compactly generated totally disconnected locally compact group. Then
\begin{itemize}
\item $G$ is geometrically amenable if and only if it admits a $1$-Lipschitz proper action on a locally finite amenable graph; and
\item $\Fol^\ast G=0$ if and only if $G$ admits a contingently $1$-Lipschitz proper action on a locally finite amenable graph.
\end{itemize}
\end{theorem}
We actually prove a slightly more detailed result than \cref{thm:geo.amen<->action.simple}, which we state below as \cref{thm:geo.amen<->action}.

\begin{remark}
\cref{lem:amen-unimod} implies that geometric amenabillity is stronger than amenability. However, it is not clear whether $\Fol^\ast G=0$ is a stronger property than $\Fol G=0$. Clearly the two conditions coincide when $G$ is unimodular. If $G$ is non-unimodular and amenable then $\Fol G=0$. On the other hand, it is easy to see that $\Fol^\ast G$ is not $0$ if the modular homomorphism has discrete image; this is the case, for example, in the affine group over $\Q_p$, where the image of the modular homomorphism is the powers of $p$. Note that \cref{thm:deltaDense.simple} shows that if $G$ is amenable, then the converse holds as well. We do not know what happens if $G$ is neither unimodular nor amenable.
\end{remark}

\begin{que}\label{qu:dense}
If $\Fol G=0$ and $\Delta_G$ has dense image in $\R^*_+$, must it be the case that $\Fol^\ast G=0$?
\end{que}

\subsection*{The space of $G$-transitive graphs} Let $ \mathfrak{G}$ be the set of  isomorphism classes of locally finite vertex-transitive graphs. Given a compactly generated totally disconnected locally compact group $G$, we define $\mathfrak{G}(G)$ to be the subset of $\mathfrak{G}$ consisting of graphs admitting a proper transitive action of $G$.
We then define $\h_G=\inf_{\Gamma\in \mathfrak{G}(G)}\h(\Gamma)$, where $\h(\Gamma)$ is the Cheeger constant of $\Gamma$ as above. This allows us to formulate the following refinement of  \cref{prop:almost.amen.trans.seq}.
 \begin{theorem}\label{prop:Folstar.versus.h}
Suppose $G$ is a compactly generated totally disconnected locally compact group. Then $\h_G=\Fol^\ast G$.
\end{theorem}
We recall that $\mathfrak{G}$ comes with a natural topology, obtained from the following distance:
we say that two graphs $\Gamma,\Gamma'\in \mathfrak{G}$ are at distance at most $2^{-n}$ if  their balls of radius $n$ are isomorphic. We observe that two graphs with different degrees are at distance $1$ apart. Moreover a standard compactness argument shows that two graphs at distance $0$ must be isomorphic, so that this indeed defines a distance on  $\mathfrak{G}$. Write $\overline{\mathfrak{G}(G)}$ for the closure of  $\mathfrak{G}(G)$ in $\mathfrak{G}$ for this topology.
\begin{theorem}\label{thm:h.semi-continuous}
The map $\Gamma\mapsto\h_\Gamma$ is upper semicontinuous on $\mathfrak{G}$. In particular, if $G$ is a compactly generated totally disconnected locally compact group such that $\overline{\mathfrak{G}(G)}$ contains an amenable graph, then $\h_G=0$.
\end{theorem}
We prove \cref{prop:Folstar.versus.h,thm:h.semi-continuous} in \cref{sec:G-trans}.

\medskip

For every $k\in \N$, let $\mathfrak{G}_k$ be the set of  isomorphism classes of vertex-transitive graphs of degree at most $k$, and let 
$\mathfrak{G}_k(G)=\mathfrak{G}(G)\cap \mathfrak{G}_k$. It is natural to consider the quantity $\h_{G,k}=\inf_{\Gamma\in \mathfrak{G}_k(G)}\h(\Gamma)$.

\begin{que}
Can we have $\h_{G,k}>0$ for all $k$ but $\h_G=0$?
\end{que}
We strongly expect the answer to be positive although we do not currently have an example.
\begin{que}
Does $\h_{G,k}=0$ for some $k$ imply that $\overline{\mathfrak{G}(G)}$ contains an amenable graph?
\end{que}

\section{The Soardi--Woess--Salvatori theorem}\label{sec:sw}
In this section we prove \cref{thm:sw}.
Our proof consists of combining \cref{lem:amen-unimod} with the following two results, the second of which is similar to a reduction appearing in \cite{salvatori} and \cite[Lemma 3.10]{blps}.
\begin{prop}\label{prop:trans}
Suppose $\Gamma$ is a connected, locally finite vertex-transitive graph, and $G$ is a locally compact group admitting a proper transitive action on $\Gamma$. Then $\Gamma$ is amenable if and only if $G$ is geometrically amenable.
\end{prop}
\begin{lemma}\label{prop:reduc}
Suppose $\Gamma$ is a connected, locally finite quasitransitive graph, and $G$ is a locally compact group admitting a proper quasitransitive action on $\Gamma$. Then there exists a connected, locally finite vertex-transitive graph $\Gamma'$ quasi-isometric to $\Gamma$, and a compact normal subgroup $H\lhd G$ such that $G/H$ acts properly transitively on $\Gamma'$.
\end{lemma}
Given these results, it is straightforward to deduce \cref{thm:sw}, as follows.
\begin{proof}[Proof of \cref{thm:sw}]
Let $\Gamma'$ be the graph and $H\lhd G$ be the compact normal subgroup given by \cref{prop:reduc}. Since $\Gamma$ and $\Gamma'$ are quasi-isometric, either both are amenable or neither is \cite[Theorem 18.13]{drutu-kapovich}. Moreover, since $H$ is compact, $G$ is amenable if and only if $G/H$ is amenable, and unimodular if and only if $G/H$ is unimodular, and hence, by \cref{lem:amen-unimod}, geometrically amenable if and only if $G'$ is geometrically amenable. The theorem therefore follows from applying \cref{prop:trans} to $\Gamma'$ and $G/H$.
\end{proof}

All that remains, then, is to prove \cref{prop:trans,prop:reduc}. We start with the following result, which is basically the key reason why geometric amenability of a group relates to amenability of a graph it acts on transitively. Here, and throughout this paper, given a group $G$ acting on a graph $\Gamma$, and a vertex $o\in\Gamma$ and a subset $X\subseteq\Gamma$ of vertices, we write $G_o$ for the stabiliser of $o$ in $G$, and
\[
G_{o\to X}=\{g\in G:g\cdot o\in X\}.
\]
Moreover, given a subset $X$ of a graph $\Gamma$ and a natural number $r$, we write $[X]_r=\{y\in\Gamma:d(y,X)\le r\}$ for the $r$-neighbourhood of $X$, and $\partial_rX=[X]_r\setminus X$ for the $r$-exterior boundary of $X$.
\begin{prop}\label{lem:S.gen}
Suppose $\Gamma$ is a connected, locally finite vertex-transitive graph, and $G$ is a locally compact group admitting a proper transitive action on $\Gamma$. Let $o\in\Gamma$, and let $S=\{g\in G:d(g\cdot o,o)\le1\}$. Then $S$ is a symmetric compact open generating set for $G$, 
and for every subset $X\subseteq\Gamma$ the set $G_{o\to X}$ is compact and open and satisfies
\begin{equation}\label{eq:pullback.mu}
\mu(G_{o\to X})=|X|\cdot\mu(G_o),
\end{equation}
and more generally
\begin{equation}\label{eq:mu.boundary}
\mu(G_{o\to X}S^r)=|[X]_r|\cdot\mu(G_o)
\end{equation}
for every $r\in\N$.
\end{prop}
\begin{proof}
The first part is essentially \cite[Lemma 3]{woess}. To see that $S$ is symmetric, note that
\[
d(g^{-1}\cdot o,o)=d(g^{-1}\cdot o,g^{-1}g\cdot o)=d(g\cdot o,o).
\]
We will prove by induction on $n$ that $d(g\cdot o,o)\le n$ for a given $n\ge1$ if and only if $g\in S^n$, which implies in particular that $S$ generates $G$. The base case $n=1$ is true by definition, whilst for $n\ge2$ we have
\begin{align*}
d(g\cdot o,o)\le n&\iff d(g\cdot o,x)\le 1\text{ for some $x$ with $d(x,o)\le n-1$}\\
   &\iff d(g\cdot o,h\cdot o)\le 1\text{ for some }h\in S^{n-1}&&\text{(by induction)}\\
   &\iff h^{-1}g\in S\text{ for some }h\in S^{n-1}&&\text{(by the $n=1$ case)}\\
   &\iff g\in S^n,
\end{align*}
as claimed.

By transitivity of the action, we may pick, for each $x\in\Gamma$, an automorphism $g_x\in G$ such that $g_x\cdot o=x$. Note then that
\[
G_{o\to X}=\bigcup_{x\in X}g_xG_o
\]
for an arbitrary subset $X\subseteq\Gamma$, which immediately implies \eqref{eq:pullback.mu}. Furthermore, $G_o$ is compact by properness and open by continuity, so this also means that $G_{o\to X}$ is compact and open whenever $X$ is finite, and in particular that $S$ is compact and open, as required.

Finally, for every $g\in G$ we have
\begin{align*}
g\in G_{o\to X}S^r&\iff\text{there exists $q\in G_{o\to X}$ such that $d(q^{-1}g\cdot o,o)\le r$}\\
    &\iff\text{there exists $q\in G_{o\to X}$ such that $d(g\cdot o,q\cdot o)\le r$}\\
    &\iff g\cdot o\in[X]_r\\
    &\iff g\in G_{o\to[X]_r},
\end{align*}
and so \eqref{eq:mu.boundary} follows from \eqref{eq:pullback.mu}.
\end{proof}

\begin{proof}[Proof of \cref{prop:trans}]
First, suppose that $\Gamma$ is amenable, and let $(A_n)_{n=1}^\infty$ be a sequence of finite subsets of $\Gamma$ such that $|\partial A_n|/|A_n|\to0$. Since $\Gamma$ is transitive, and hence has uniformly bounded degrees, we in fact have that $|\partial_r A_n|/|A_n|\to0$ for all $r$. Let $K$ be a compact subset of $G$. \cref{lem:S.gen} says that $S$ is a symmetric open generating set for $G$, so we have $K\subseteq S^r$ for some $r$. This in turn implies that $G_{o\to A_n}K\subseteq G_{o\to A_n}S^r$ for each $n$, and hence, by \cref{lem:S.gen}, that
\[
\frac{\mu(G_{o\to A_n}K)}{\mu(G_{o\to A_n})}\le\frac{\mu(G_{o\to A_n}S^r)}{\mu(G_{o\to A_n})}=\frac{|[A_n]_r|}{|A_n|}\to1.
\]

Conversely, if $G$ is geometrically amenable then since $S$ is compact there exists a sequence $(U_n)_{n=1}^\infty$ of compact subsets of positive measure in $G$ such that $\mu(U_nS)/\mu(S)\to1$.
Since $G_o\subseteq S$, this implies in particular that
$\mu(U_nS)/\mu(U_nG_o)\to1$.
Using the fact that $SG_o=S$ and $S=S^{-1}$, we have $G_oS=(SG_o)^{-1}=S$, and so we deduce further that $\mu(U_nG_oS)/\mu(U_nG_o)\to1$. Since $U_nG_o=G_{o\to U_n\cdot o}$, this combines with \cref{lem:S.gen} to show that $|\partial(U_n\cdot o)|/|U_n\cdot o|\to0$, and so $\Gamma$ is amenable.
\end{proof}

\begin{proof}[Proof of \cref{prop:reduc}]
It is well known that if $G$ has $n$ orbits, then if we fix one of these orbits $V$, and define $E=\{(x,y)\in V\times V:1\le d(x,y)\le2n\}$, the resulting graph $\Gamma'=(V,E)$ is connected, locally finite and quasi-isometric to $\Gamma$ (see e.g. the proof of \cite[Proposition 2.13]{hutchcroft-tointon}). We claim we may take $H$ to be the kernel of the homomorphism $G\to\Aut(\Gamma')$ given by the restriction to $V$ of the $G$-action on $\Gamma$. Indeed, the action of $G/H$ on $\Gamma'$ induced by this homomorphism is transitive by definition, whilst $H=\{g\in G:g\cdot v=v\text{ for all }v\in V\}$ is a closed subset of a vertex stabiliser, and hence compact.
\end{proof}

\section{Equivalent formulations of amenability and geometric amenability}

It is well known that in order to decide whether a locally compact group is amenable it suffices to consider the individual elements of a single compact generating set, as follows.
\begin{lemma}\label{lem:amen.gen.set}
Suppose $G$ is a locally compact group, and $S\subseteq G$ is a compact set generating $G$ as a semigroup. Then $G$ is amenable if and only if for each $\eps>0$ there exists a compact set $F\subseteq G$ of positive measure such that $\sup_{s\in S}\mu(sF\vartriangle F)/\mu(F)\le\eps$.
\end{lemma}
Although this is well known, we have not been able to locate a convenient self-contained reference, so we provide a proof.
\begin{proof}
If $G$ is amenable then by definition there exist compact sets $(F_n)_{n=1}^\infty$ of positive measure such that $\mu((S\cup S^{-1})F_n\setminus F_n)/\mu(F_n)\to0$. In particular, for each $s\in S$ we have $\mu(sF_n\setminus F_n)/\mu(F_n)\to0$ and $\mu(F_n\setminus sF_n)/\mu(F_n)=\mu(s(s^{-1}F_n\setminus F_n))/\mu(F_n)=\mu(s^{-1}F_n\setminus F_n)/\mu(F_n)\to0$.

Conversely, suppose $(F_n)_{n=1}^\infty$ is a sequence of compact subsets of positive measure in $G$ such that $\mu(sF_n\vartriangle F_n)/\mu(F_n)\le\frac1n$ for all $s\in S$. We claim more generally that $\mu(gF_n\vartriangle F_n)/\mu(F_n)\to0$ for all $g\in G$. Indeed, since $S$ generates $G$ as a semigroup, an arbitrary element $g\in G$ can be written in the form $g=s_1\cdots s_m$ with $s_i\in S$, and then using the well-known and easily verified fact that $\mu(A\vartriangle B)$ satisfies the triangle inequality we obtain
\[
\begin{split}
\mu(gF_n\vartriangle F_n)\le\mu(s_1\cdots s_{m-1}(s_mF_n\vartriangle F_n))+\mu(s_1\cdots s_{m-2}(s_{m-1}F_n\vartriangle F_n))+\cdots+\mu(s_1F_n\vartriangle F_n)\\
    \le\frac mn\mu(F_n)\to0.\qquad\qquad\qquad\qquad\qquad\qquad\qquad\qquad\qquad\qquad\qquad\qquad
\end{split}
\]
The implication (iv) $\implies$ (v) of \cite[Theorem G.3.1]{BHV} then implies that $L^\infty(G)$ admits an invariant mean, so that $G$ is \emph{amenable} in the sense of \cite{emerson-greenleaf}, and then the implication (amenable) $\implies$ (A) proved in \cite[\S1.2]{emerson-greenleaf} implies that $G$ is amenable in our sense by the remarks at the end of \cite[\S1.2]{emerson-greenleaf}.
\end{proof}
In the present work we need the following analogous result for geometric amenability.
\begin{lemma}\label{lem:geom.amen.gen.set}
Suppose $G$ is a locally compact group, and $S\subseteq G$ is a compact set generating $G$ as a semigroup. Then $G$ is geometrically amenable if and only if for each $\eps>0$ there exists a compact set $F\subseteq G$ of positive measure such that $\sup_{s\in S}\mu(Fs\vartriangle F)/\mu(F)\le\eps$.
\end{lemma}
\begin{proof}
We first show that if $G$ is not unimodular then neither condition holds. It is convenient to prove the contrapositive. If $G$ is geometrically amenable this is immediate from \cref{lem:amen-unimod}. On the other hand, if $(F_n)_{n=1}^\infty$ is a sequence of compact subsets of positive measure in $G$ such that $\mu(F_ns\vartriangle F_n)/\mu(F_n)\le\frac1n$ for all $s\in S$, then $\mu(F_ns)\le(1+\frac1n)\mu(F_n)$ for every $s\in S$ and $n\in\N$, hence $\Delta_G(s)\le1+\frac1n$ for every $s\in S$ and $n\in\N$, and hence $\Delta_G(s)\le1$ for every $s\in S$. Since $S$ generates $G$ as a semigroup, this implies that $\Delta_G\equiv1$ on $G$ as claimed.

We may therefore assume that $G$ is unimodular, and in particular that $\mu$ is symmetric. By \cref{lem:amen-unimod}, $G$ is then geometrically amenable if and only if it is amenable; by \cref{lem:amen.gen.set}, $G$ is amenable if and only if for each $\eps>0$ there exists a compact set $F\subseteq G$ of positive measure such that $\sup_{t\in S^{-1}}\mu(sF\vartriangle F)/\mu(F)\le\eps$; and by symmetry of $\mu$, this occurs if and only if for each $\eps>0$ there exists a compact set $F^{-1}\subseteq G$ of positive measure such that $\sup_{s\in S}\mu(F^{-1}s\vartriangle F^{-1})/\mu(F^{-1})\le\eps$.
\end{proof}

\section{Lipschitz proper actions and geometric amenability}

In this section we generalise \cref{thm:LipActionOnAmenable} to graphs of unbounded degree. This generalisation necessitates a further definition: we will say that a locally finite graph is $r$-amenable for $r\geq 1$, and for all $\eps>0$, there exists a finite set of vertices $F$ such that $|\partial_rF|/|F|\leq \eps$. When $r=1$, we simply recover the usual notion of amenability. Note that if the graph has uniformly bounded degrees then amenability implies $r$-amenability for all $r$.
\begin{prop}\label{prop:LipActionOnAmenable}
Let $r\in\N$. Suppose $G$ is a compactly generated, totally disconnected, locally compact group acting $r$-Lipschitz properly on a locally finite $r$-amenable graph $\Gamma$. Then $G$ is geometrically amenable.
\end{prop}

Before proving \cref{prop:LipActionOnAmenable} we present two lemmas.

\begin{lemma}\label{lem:almostLipImpliesCoarseConnected}
Suppose $G$ is a compactly generated locally compact group acting properly on a locally finite graph $\Gamma$. Then the action of $G$ on $\Gamma$ is contingently $1$-Lipschitz if and only if every subgraph induced by an orbit of $G$ on $\Gamma$ is connected.
\end{lemma}
\begin{proof}
Suppose first that the action is contingently $1$-Lipschitz. Given $x\in\Gamma$, there therefore exists a generating set $S_x$ such that $g\mapsto g\cdot x$ is a $1$-Lipschitz map from the Cayley graph $(G,S_x)$ to $X$. The fact that $(G,S_x)$ is connected implies that the range of this map is connected as well, hence the orbit of $x$ is connected.

Conversely, suppose that the orbit of $x$ is connected. Then the subgraph $\Gamma_x$ induced by $G\cdot x$ is a connected locally finite vertex-transitive graph, and by \cref{lem:S.gen} the set $S_x=\{g\in G:d_{\Gamma_x}(g\cdot x,x)\le1\}$ is a symmetric compact open generating set for $G$, with respect to which $g\mapsto g\cdot x$ is trivially $1$-Lipschitz.
\end{proof}

\begin{lemma}\label{lem:quantitativeVT}
Suppose $G$ is a locally compact group with Haar measure $\mu$ acting transitively on a locally finite graph $\Gamma$. Suppose further that this action is $1$-Lipschitz with respect to some compact symmetric generating set $S$ for $G$. Then for every finite subset $F\subseteq\Gamma$ there exists a compact open subset $A\subseteq G$ such that
\[
\frac{\mu(AS\setminus A)}{\mu(A)}\le\frac{|\partial F|}{|F|}.
\]
\end{lemma}
\begin{proof}
Let $x\in\Gamma$ such that the orbit map $(G,S)\to\Gamma$, $g\mapsto g\cdot x$ is $1$-Lipschitz. Let $\hat{S}=\{g\in G:d(x,g\cdot x)\le1\}$. \cref{lem:S.gen} implies that $G_{x\to F}$ is a compact open set satisfying
\[
\frac{\mu(G_{x\to F}\hat S\setminus G_{x\to F})}{\mu(G_{x\to F})}=\frac{|\partial F|}{|F|}.
\]
Moreover, the fact that the orbit map is $1$-Lipschitz implies that $d(x,gx)\leq |g|_S$ for all $g\in G$. Applying this to $g\in S$ implies that $S\subseteq\hat{S}$, hence that $G_{x\to F}S\subseteq G_{x\to F}\hat S$, and hence that $\mu(G_{x\to F}S\setminus G_{x\to F})\le\mu(G_{x\to F}\hat S\setminus G_{x\to F})$, so that we may take $A=G_{x\to F}$.
\end{proof}

\begin{proof}[Proof of \cref{prop:LipActionOnAmenable}]
Upon adding edges between all pairs of vertices at distance at most $r$, we may assume that $r=1$. 
By \cref{lem:almostLipImpliesCoarseConnected}, the subgraph induced by each orbit is connected. Pick a compact symmetric generating set $S$, and a vertex $z$ in each orbit such that the orbit map $g\to g\cdot z$ is $1$-Lipschitz with respect to $S$, 
and write $Z$ for the set of such $z$. For every $z\in Z$, denote by $Y_z$ the graph induced by the orbit of $z$, and let $Y=\bigsqcup_{z\in Z} Y_z$.  In other words, $Y$ is obtained from $\Gamma$  by removing all edges joining different orbits. Note that each $Y_z$ is a vertex-transitive graph such that the orbit map $(G,S)\to Y_z$, $g\mapsto g\cdot z$ is $1$-Lipschitz.

Let $\eps>0$. By amenability of $\Gamma$ there exists $F\subseteq\Gamma$ be such that $|\partial_F|/|F|\le\eps$. Write $\partial^YF$ for the external boundary of $F$ in $Y$, noting that
$\partial^YF\subseteq\partial F$ and that $\partial^YF=\bigsqcup_{z\in Z} \partial^YF_z$, where $F_z=F\cap G\cdot z$.
By the pigeonhole principle, there exists $z\in Z$,  such that $|\partial^Y_1F_z|/|F_z|\leq |\partial^YF|/|F|\leq  |\partial F|/|F|\leq \eps$.
Applying Lemma \ref{lem:quantitativeVT} to the action of $G$ on the vertex-transitive graph $Y_z$, we therefore conclude that there exists a compact open set $A\subseteq G$ such that $\mu(AS\setminus A)/\mu(A)\le\eps$. In particular, this implies that
\[
\begin{split}
\sup_{s\in S}\frac{\mu(As\vartriangle A)}{\mu(A)}=\sup_{s\in S}\frac{\mu(As\setminus A)+\mu((As^{-1}\setminus A)s)}{\mu(A)}\qquad\qquad\qquad\qquad\qquad\\
         \le\big(1+\sup_{s\in S}\Delta_G(s)\big)\frac{\mu(AS\setminus A)}{\mu(A)}\le\big(1+\sup_{s\in S}\Delta_G(s)\big)\eps,
\end{split}
\]
so that $G$ is geometrically amenable by \cref{lem:geom.amen.gen.set}.
\end{proof}

\section{The space of $G$-transitive graphs}\label{sec:G-trans}

In this section we prove \cref{prop:Folstar.versus.h,thm:h.semi-continuous}.
\begin{proof}[Proof of \cref{prop:Folstar.versus.h}]
Let us start proving that $\Fol^\ast G\leq \h_G$. Assume the existence of a sequence $\Gamma_n$ of proper $G$-transitive graphs, and of finite subsets $A_n$ such that $|\partial A_n|/|A_n|\to\Fol^\ast G$. Let $o_n$ be some vertex in $\Gamma_n$, denote by $K_n$ the stabilizer of $o_n$ in $G$, and let $S_n=\{g\in G:d(g\cdot o_n,o_n)\le1\}$.  By \cref{lem:S.gen}, $S_n$ is a compact open generating subset of $G$, and we have that \[\mu(G_{o_n\to A_n}S_n)/\mu(G_{o_n\to A_n})=|[A_n]_1|/|A_n|.\] Hence we deduce that $\Fol^\ast G\leq \h_G$. 

Observe that compact subsets of the form $\{g\in G:d(g\cdot o,o)\le1\}$ for a proper $G$-transitive pointed graph $(\Gamma,o)$ satisfy $S=KSK$ for some compact open subgroup $K$. The main point of the converse inequality is to show that in the definition of $\Fol^\ast G$, we only need to take the infimum over such generating subsets.
Precisely: given $\eps>0$, two compact subsets $S$ and $F$ such that $\mu(F)>0$, we claim that there exists a compact open subgroup $K$ such that $\mu(FKSK\setminus FS)\leq \eps\mu(FS)$.
Since the Haar measure is regular, one can find an open neighbourhood of the identity $W$ of $G$ such that $\mu(FSW\setminus FS)\leq \eps\mu(FS)$.
We now have to find a compact open subgroup $K$ such that $FKSK\subseteq FSW$.
Using that the multiplication is continuous, we see that there exist neighbourhoods $S'$ and $F'$ of $S$ and $F$ such that $F'S'\subseteq FSW$. 
Now because $S$ and $F$ are compact, such neighbourhoods can be taken to be of the form $S'=SU$ and $F'=FU$ for some neighbourhood $U$ of the neutral element in $G$. But since $G$ is totally disconnected, $U$ contains some compact open subgroup $K$. So we finally have that $FKSK\subseteq FSW$, and so the claim follows.

We are now ready to prove that $\h_G\leq \Fol^\ast G$. Consider a sequence of compact symmetric generating subsets $S_n$ and a sequence of compact subsets of positive measure $F_n$ such that $\mu(F_nS_n\setminus F_nS_n)/\mu(F_n)$ tends to $\Fol^\ast G$.
By our claim, we deduce the existence of a sequence of compact open subgroups $K_n$ such that $\mu(F_nK_nS_nK_n\setminus F_nS_n)/\mu(F_n)$ tends to zero. Hence, $\mu(F_nK_nS_nK_n\setminus F_n)/\mu(F_n)$ tends to $\Fol^\ast G$. 
Letting $S'_n=K_nS_nK_n$, we deduce that \[\liminf\mu(F_nK_n S'_n\setminus F_nK_n)/\mu(F_nK_n)\leq \Fol^\ast G.\] 

Consider now the Cayley--Abels graph $\Gamma_n$ obtained as right quotient of $(G_n,S'_n)$ by $K_n$ (see \cite[Proposition 2.E.9]{CH} for the definition of a Cayley--Abels graph, originally due to Abels \cite{ab}). Denote by $\pi_n$ the projection modulo $K_n$. Let $A_n=\pi_n(F_n)$. Since $F_nK_n$ and $F_nK_n S'_n\setminus F_nK_n$ are unions of $K_n$ left cosets, we have 
\[\partial A_n=\pi(F_nK_n S'_n)\setminus \pi(F_nK_n)=\pi(F_nK_n S'_n\setminus F_nK_n).\]
Now, because $\mu$ is left-invariant, we deduce that
$|\partial A_n|\mu(K_n)=\mu(F_nK_n S'_n\setminus F_nK_n)$, and $|A_n|\mu(K_n)=\mu(F_nK_n)$. Hence $\liminf|\partial A_n|/|A_n|\leq \Fol^\ast G$ as required.
\end{proof}

We now turn to the proof of Theorem \ref{thm:h.semi-continuous}. 
Recall that the isoperimetric profile of a  graph $\Gamma$ is defined via
\[j_\Gamma(n)=\inf_{|A|\leq n}\left\{\frac{|\partial A|}{|A|}\right\}.\]
\begin{prop}\label{prop:isopcont}
For every $n\in \N$, the map $\Gamma\mapsto j_\Gamma(n)$ is continuous on $\mathfrak{G}$. 
\end{prop}
\begin{proof}
Let $A$ be a finite subset of a graph $\Gamma$, and assume that  $A=A_1\sqcup A_2$ are such that $d(A_1,A_2)\geq 3$. Then we have $\partial A=\partial A_1\sqcup \partial A_2$. We therefore deduce that $|\partial A|=|\partial A_1|+|\partial A_2|$. Hence we deduce that 
\[\frac{|\partial A|}{|A|}\leq \min\left\{\frac{|\partial A_1|}{|A_1|},\frac{|\partial A|}{|A_2|}\right\}.\]
Hence $j_\Gamma(n)$ is attained on subsets $A$ that are $2$-connected: meaning that every pair of vertices $x,y$ can be joined by a chain of vertices $x=x_0,\ldots, x_k=y$ such that $d(x_i,x_{i+1})\leq 2$. 
Since such sets are contained in a ball of radius $2n$, we have that $j_\Gamma(n)=j_{\Gamma'}(n)$ as soon as $d(\Gamma,\Gamma')\leq 2^{-2n-1}$, meaning that the balls of radius $n+1$ of these two graphs coincide. This proves the proposition.
\end{proof}
\begin{proof}[Proof of Theorem \ref{thm:h.semi-continuous}]
Note that $\h_{\Gamma}=\inf_n j_\Gamma(n)$, so $\Gamma\mapsto \h_{\Gamma}$ is an infimum of continuous functions by \cref{prop:isopcont}, hence is upper semicontinuous.
\end{proof}

\section{$\Fol^\ast G$}

In this section we prove our various results about $\Fol^\ast G$, which recall we defined via
\[
\Fol^\ast G=\inf_S\inf_U\frac{\mu(US\setminus U)}{\mu(U)},
\]
where the infima are over all compact symmetric generating sets $S$ for $G$ and all compact subsets $U\subseteq G$. We start by observing that one direction of \cref{thm:deltaDense.simple} does not need the amenability assumption.
\begin{prop}\label{prop:dense}
Suppose $G$ is a non-unimodular locally compact group, and that $\Delta_G(G)$ is discrete. Then $G$ is uniformly geometrically non-amenable, i.e.\ $\Fol^\ast G>0$.
\end{prop}
Note that $\Delta_G(G)$ is either discrete or dense, so that this really does prove one direction of \cref{thm:deltaDense.simple}.

\begin{proof}[Proof of \cref{prop:dense}]
The fact that $\Delta_G(G)$ is non-trivial and discrete implies that it is cyclic, so that we may fix a generator $t>1$. Note, then, that every symmetric generating set $S$ for $G$ must contain an element $s$ such that $\Delta_G(s)\ge t$, so that for every compact subset $F$ of positive measure we have
\[
\mu(FS\setminus F)\ge\mu(Fs\setminus F)\ge\mu(Fs)-\mu(F)=(t-1)\mu(F).
\]
\end{proof}

\begin{examples}
Here is an example of a group with discrete image of the modular homomorphism: the affine group $\Aff(\Q_p)=\Q_p\rtimes \Z$ (for which $t=p$ in the proof of \cref{prop:dense}). On the other hand, modular homomorphism of the direct product $\Aff(\Q_p)\times\Aff(\Q_q)$ has dense image whenever $p$ and $q$ are not powers of a common integer. It turns out that this group is not uniformly geometrically non-amenable, as shown by the following proposition.
\end{examples}

\begin{prop}\label{prop:deltaDenseAlmostAm}
Suppose that $G$ is an amenable, non-unimodular compactly generated totally disconnected locally compact group. Then $\Fol^\ast G=0$ if and only if the image of the modular homomorphism is dense in $\R^*_+$. Moreover, if $\Delta_G$ is split and $\Fol^\ast G=0$,  then  there exists a sequence $(S_n)_{n=1}^\infty$ of generating subsets of the form $S_n=K\sqcup T_n$, where $K$ is a fixed compact subset of $\ker\Delta_G$ and $(T_n)_{n=1}^\infty$ is a sequence of finite subsets of bounded cardinality
 satisfying
\[
\frac{\mu(F_nS_n\setminus F_n)}{\mu(F_n)}\to0
\]
for some sequence $(F_n)_{n=1}^\infty$ of compact subsets of $G$.
\end{prop}

We start with a lemma to help us construct the required generating sets $S_n$.

\begin{lemma}\label{lem:split.gen}
Suppose that $G$ is a compactly generated locally compact group, that $1\to N\to G\to Q\to 1$ is a short exact sequence of locally compact groups. Then given any compact generating set $U$ for $Q$, and any relatively compact lift $V$ of $U$ in $G$, there exists a compact symmetric subset $R$ of $N$ such that $R\cup V$ generates $G$. Moreover, if the sequence is split then there exists a fixed compact symmetric subset $K\subseteq N$ such that given any compact generating set $U$ for $Q$, the set $K\cup U$ generates $G$.
\end{lemma}
\begin{proof}Write $\pi:G\to Q$ for the quotient homomorphism. Let $S$ be a compact symmetric generating set for $G$, and fix $n\in\N$ so that $\pi(S)\subseteq U^n$. We claim that we may take $R$ to be the (compact) closure of $(SV^{-n}\cup V^nS)\cap N$.
Indeed, this set is symmetric by definition, and given $s$ in $S$, there exists $g$ in $V^n$ such that $sg^{-1}\in N$, hence $sg^{-1}\in R$, and hence $s=sg^{-1}g\in RV^n$. 

If the sequence is split then let $V=\pi(S)$, and let $K$ be the (compact) closure of $(SV^{-1}\cup VS)\cap N$, noting that $K\cup V$ generates $G$ by the previous paragraph. Then if $U$ is a compact generating set for $Q$, we have $V\subseteq U^m$ for some $m\in\N$, so that $K\cup U$ generates $G$ as required.
\end{proof}

\begin{proof}[Proof of \cref{prop:deltaDenseAlmostAm}]
As noted above, $\Delta_G(G)$ is either discrete or dense in $\R^*_+$, so \cref{prop:dense} implies that it is dense if $\Fol^\ast G=0$. It therefore remains to prove that $\Fol^\ast G=0$ assuming that $\Delta_G(G)$ is dense.

Being totally disconnected, $G$ has a compact open subgroup \cite[Theorem 7.7]{hew-ross}.
The image of this subgroup under $\Delta_G$ is a compact subrgoup of $\R$, and hence trivial, so $\Delta_G$ factors through a discrete, and hence finitely generated, quotient. Since $\R^*_+$ is abelian and torsion-free, $\Delta_G$ factors through a finitely generated torsion-free abelian quotient $G\to A\cong\Z^d$. Write $\pi:G\to A$ for the quotient homomorphism, and let $\Delta':A\to \R^*_+$ be an injective homomorphism such that $\Delta_G=\Delta'\circ\pi$. If $\Delta_G$ is split, let $K$ be the symmetric compact set given by \cref{lem:split.gen}.

We claim that $V\cap\Delta'(A)$ spans $\Delta'(A)$ for every neighbourhood $V$ of $1\in \R^*_+$. It is more convenient to work additively:\ consider $\delta'=\log \circ\Delta':A\to (\R,+)$, and take $W=\log V$ of the form $W=(-t,t)$ for some $t>0$. Let $a\in \delta'(A)$. By density of $\Delta'(A)=\Delta_G(G)$, there exists a sequence $0=a_1,a_2,\ldots, a_m=a$ of elements of $\delta'(A)$ such that $|a_{i+1}-a_i|< t$ for each $i$, and hence $a_{i+1}-a_{i}\in W\cap \delta'(A)$ for each $i$, so that $\delta'(A)\cap W$ generates $\delta'(A)$ as claimed. In particular, for all such $V$ we can find a basis $(x_1,\ldots, x_d)$ of $A\cong \Z^d$ whose image under $\Delta'$ lies in $V$.

Let $n\in\N$.
Note that for any $t>0$ close enough to $1$, we have 
$t^{n+1}\leq\frac1n(\sum_{i=-n}^{n}t^i)$ (as for $t=1$ we have strict inequality). 
We may therefore pick $V_n\subseteq(\frac12,2)$ small enough such that
\begin{equation}\label{eq:Vsmall}
v^{n+1}\leq \frac1n\left(\sum_{i=-n}^{n}v^i\right)\qquad\qquad\forall\,v\in V_n,
\end{equation}
and fix some basis $(x_1,\ldots, x_d)$ for $A$ whose image under $\Delta'$ lies in $V_n$. From now on, we identify $A$ with $\Z^d$ via this basis.

Let $j:A\to G$ be a cross section of $\pi$ satisfying $j(0)=1$, chosen to be the natural embedding $A\hookrightarrow\ker\pi\rtimes A$ if $\Delta_G$ is split, and set $T_n=\{j(x_1)^{\pm1}, \ldots, j(x_d)^{\pm1}\}$. 
\cref{lem:split.gen} then implies that there exists a compact subset 
$R_n\subseteq\ker\pi$ such that $S_n=R_n\cup T_n$ forms a symmetric compact generating subset of $G$. Moreover, if $\Delta_G$ is split then by definition of $K$ we may take $S_n=K\sqcup T_n$ for each $n$, as required.

Let $K_n=j([-n,n]^d\cap\Z^d)\subseteq G$. 
Define 
$X_n=\bigcup_{k\in K_n}kR_nk^{-1}$, which is a compact subset of $\ker\pi$. 
Finally, for all $a\in K_n$, we have $\pi(a)=(a_1,\ldots, a_d)$ with $a_i\in [-n,n]$, and for each permutation $\sigma\in \mathfrak{S}(d)$ the elements $a$ and $j(x_{\sigma(1)})^{a_{\sigma(1)}}\ldots j(x_{\sigma(d)})^{a_{\sigma(d)}}$ differ by an element of $\ker\pi$. Let $L_n\subseteq \ker\pi$ be the (finite) subset of all such elements and their inverses.

We now form a compact subset $M_n\subseteq\ker\pi$ by setting $M_n=L'_n\cup L_n\cup X_n$. 
Note that since $\ker\pi$ is open in $G$, the restriction of $\mu$ to $\ker\pi$ is a Haar measure on $\ker\pi$. Moreover, since $\ker\pi$ is the kernel of $\Delta_G$, it is unimodular, and since it is a closed subgroup of $G$, it is also amenable. Hence $\ker\pi$ is geometrically amenable by \cref{lem:amen-unimod}, and so there exists a compact subset $Q_n\subseteq\ker\pi$ such that 
\[
\mu(Q_nM_n\setminus Q_n)\leq\textstyle\frac1n\mu(Q_n).
\]
Let $F_n=Q_nK_n$.

We claim that 
\begin{equation}\label{eq:FFolner}
\mu(F_nS_n\setminus F_n)\le O_d(\textstyle\frac1n)\mu(F_n),
\end{equation}  
which will immediately finish the proof of the proposition.
To see this, first note that $F_n= \bigsqcup_{k\in K_n}Qk$, and hence
\begin{equation}\label{eq:Fn<Qn}
\mu(F_n)=\left(\sum_{k\in K_n}\Delta_G(k)\right)\mu(Q_n).
\end{equation}
Note also that by definition of $X_n$ we have 
\[
F_nR_n=Q_nK_nR_n\subseteq Q_nX_nK_n\subseteq Q_nM_nK_n= \bigsqcup_{k\in K_n}Q_nM_nk,
\]
hence
\[
F_nR_n\setminus F_n\subseteq \bigsqcup_{k\in K_n}(Q_nM_n\setminus Q_n)k,
\]
and hence
\[
\mu(F_nR_n\setminus F_n)\le\left(\sum_{k\in K_n}\Delta_G(k)\right)\mu(Q_nM_n\setminus Q_n).
\]
It then follows from \eqref{eq:Fn<Qn} and the definition of $Q_n$ that 
\begin{equation}\label{eq:FS_HFolner}
\mu(F_nR_n\setminus F_n)\le\textstyle{\frac1n}\mu(F_n).
\end{equation}

We now turn to $T_n$. By symmetry in $x_1,\ldots,x_d$ and in $x_d,x_d^{-1}$, it is enough to show that
\begin{equation}\label{eq:dense.almost.am.final.aim}
\mu(F_nj(x_d)\setminus F_n)\le O(\textstyle{\frac1n})\mu(F_n);
\end{equation}
this will combine with \eqref{eq:FS_HFolner} to prove \eqref{eq:FFolner}, as claimed.
To that end, let
\[
K'_n=\{j(x_1)^{a_1}\ldots j(x_d)^{a_d}: (a_1,\ldots,a_d)\in [-n,n]^d\cap\Z^d\},
\]
and let $F_n'=Q_nK_n'$. By definition of $L'_n$, we have $F_n\subseteq Q_nL_nK_n'$ and $F_n'\subseteq Q_nL_nK_n$. We claim that
\begin{equation}\label{eq:F.to.F'}
\mu(F_n\vartriangle F'_n)\leq O(\textstyle{\frac1n})\mu(F_n).
\end{equation}
To see this, first note that by definition of $L'_n$ we have $F_n'\subseteq Q_nL_nK_n\subseteq Q_nM_nK_n$, and hence
\[
F_n'\setminus F_n\subseteq Q_nM_nK_n\setminus Q_nK_n=\bigsqcup_{k\in K_n}(Q_nM_n\setminus Q_n)k.
\]
In particular, this implies that
\[
\mu(F_n'\setminus F_n)\le\left(\sum_{k\in K_n}\Delta_G(k)\right)\mu(Q_nM_n\setminus Q_n),
\]
so that by \eqref{eq:Fn<Qn} and definition of $Q_n$ we have $\mu(F_n'\setminus F_n)\le\frac1n\mu(F_n)$, as required. On the other hand, we similarly have $F_n\subseteq Q_nL_nK_n'$, and a similar argument then shows that $\mu(F_n\setminus F_n')\le\frac1n\mu(F_n')\le\frac2n\mu(F_n)$, giving \eqref{eq:F.to.F'} as claimed. Since $\Delta'(x_d)<2$, \eqref{eq:F.to.F'} in turn implies that
\[
\mu(F_nj(x_d)\setminus F_n'j(x_d))=\mu((F_n\setminus F_n')j(x_d))=\Delta'(x_d)\mu((F_n\setminus F_n')<\textstyle{\frac2n}\mu(F_n).
\]
which then combines with \eqref{eq:F.to.F'} to show that in order to prove \eqref{eq:dense.almost.am.final.aim}---and hence the proposition---it is enough to prove that
\begin{equation}\label{eq:Fn'j(x_d)}
\mu(F_n'j(x_d)\setminus F_n')\le\textstyle{\frac1n}\mu(F_n').
\end{equation}

Let $K_n''=\{j(x_1)^{a_1}\ldots j(x_{d-1})^{a_{d-1}}: (a_1,\ldots,a_{d-1})\in [-n,n]^{d-1}\cap\Z^{d-1}\}.$
We have 
\[
F_n'=\bigsqcup_{i=-n}^nQ_nK_n''j(x_d)^i,
\]
so that
\[
\mu(F_n')=\left(\sum_{i=-n}^n \Delta_G(j(x_d))^i\right)\mu(Q_nK''_n)=\left(\sum_{i=-n}^n \Delta'(x_d)^i\right)\mu(Q_nK_n'')
\]
and 
\[
F_n'j(x_d)\setminus F_n'\subseteq Q_nK_n''j(x_d)^{n+1},
\]
and hence
\[
\mu(F'j(x_d)\setminus F_n')\leq \Delta_G(j(x_d))^{n+1}\mu(Q_nK_n'')= \Delta'(x_d)^{n+1}\mu(Q_nK_n'').
\]
Applying \eqref{eq:Vsmall} with $v=\Delta'(x_d)$, we deduce that \eqref{eq:Fn'j(x_d)} holds as required.
\end{proof}

Finally, we prove the following slight refinement of \cref{thm:geo.amen<->action.simple}.
\begin{theorem}\label{thm:geo.amen<->action}
Suppose $G$ is a compactly generated totally disconnected locally compact group. Then
\begin{itemize}
\item $G$ is geometrically amenable if and only if it admits a $1$-Lipschitz proper action on a locally finite amenable graph; and
\item $\Fol^\ast G=0$ if and only if $G$ admits a contingently $1$-Lipschitz proper action on a locally finite amenable graph, if and only if there exists $r\in\N$ such that $G$ admits a contingently $r$-Lipschitz proper action on a locally finite $r$-amenable graph for some $r\geq 1$. 
\end{itemize}
\end{theorem}

\begin{proof}
If $G$ is geometrically amenable then by \cref{thm:amen+unimod} there exists $r\in\N$ such that $G$ admits an $r$-Lipschitz proper action on a bounded-degree amenable graph $\Gamma_0$. Then $G$ acts $1$-Lipschitz properly on the bounded-degree amenable graph obtained from $\Gamma_0$ by adding edges between all pairs of vertices at distance at most $r$. The converse is given by \cref{prop:LipActionOnAmenable}.

Now suppose that $\Fol^\ast G=0$. By \cref{prop:almost.amen.trans.seq} (which recall followed from \cref{prop:Folstar.versus.h}), there exists a sequence $\Gamma_n$ of $G$-vertex transitive proper locally finite graphs and a sequence of finite subsets $A_n\subset \Gamma_n$ such that $|\partial A_n|/|A_n|\to 0$. Consider the disjoint union $\Gamma=\bigsqcup \Gamma_n$: since the graphs induced by the orbits are connected, we deduce from \cref{lem:almostLipImpliesCoarseConnected} that the action is contingently $1$-Lipschitz. Besides, $\Gamma$ is obviously amenable, and the $G$-action on it is proper. 

Conversely, assume that $G$ admits a contingently $r$-Lipschitz proper action on a locally finite $r$-amenable graph for some $r\geq 1$. On adding edges between all pairs of points at distance at most $r$, we reduce to the case $r=1$.
Hence we can assume that $G$ admits a $1$-Lipschitz proper action on a locally finite amenable graph $\Gamma$. Removing edges from $\Gamma$ does not change the fact that it is amenable, so  we may assume that $\Gamma$ is a disjoint union of vertex-transitive proper locally finite graphs $\Gamma_n$. 
Now if $F\subseteq \Gamma$, we have $F=\bigsqcup_n F_n$, where $F_n\subseteq \Gamma_n$, and $\partial F=\bigsqcup_n \partial F_n$. By the pigeonhole principle, there exists $n$ such that $|\partial F_n|/|F_n|\leq |\partial F|/|F|$.
We deduce that if $(A_k)$ is a sequence of subsets of $\Gamma$ such that $|\partial A_k|/|A_k|\to 0$, there exists $A'_k\subseteq A_k$ such that $|\partial A'_k|/|A'_k|\to 0$ and such that $A'_k$ is contained in some $\Gamma_{n_k}$ for each $k$. Hence the sequence $\Gamma_{n_k}$ is asymptotically amenable, which implies that $\Fol^\ast G=0$ by \cref{prop:almost.amen.trans.seq}. 
\end{proof}

\section{Sol}
Recall that the group Sol is defined as the semi-direct product $\Z^2\rtimes \Z$, where the generator of $\Z$ acts on $\Z^2$ via the matrix $
 \left(\begin{array}{cc} 2 & 1  \\ 1 & 1 \end{array}\right)$.
\begin{prop}\label{prop:Sol}
Sol does not admit a continuous quasitransitive action on a locally finite non-amenable graph.
\end{prop}
\begin{proof}
Assume by contradiction that such action exists on a graph $X$. Then since the image of Sol is cocompact in the automorphism group of $X$, its closure $G$ is quasi-isometric to $X$. But because Sol is amenable, then so is $G$. Since $X$ is non amenable though, $G$ is not geometrically amenable, and so the only way this can happen is if $G$ is non-unimodular. Hence we are reduced to proving that Sol does not admit any continuous morphism with dense image $\pi:Sol\to G$, where $G$ is totally disconnected and non-unimodular. 

We start observing that since Sol is metabelian, then so is $G$, we deduce that $A:=\overline{[G,G]}$ is abelian. Recall that Sol is isomorphic to $\Z^2\rtimes \Z$, where $\Z^2$ coincides with the derived subgroup. We therefore have that $\Z^2\subset [G,G]$. On the other hand, we have $\Z\cap A=\{1\}$. Indeed, since $\overline{[G,G]}$ is abelian, this would imply that the image of Sol is virtually abelian, which would be incompatible with the fact that $G$ is non-unimodular. This implies that $Sol\cap A=\Z^2$, and therefore that $\Z^2$ is dense in $A$.

Moreover, the image of $G$  in $G^{ab}:=G/A$ coincides with the image of $\Z$, which is dense. We have two possibilities: either $G^{ab}$ is compact, which again would be at odd with the fact that $G$ is non-unimodular, or it is discrete and therefore isomorphic to $\Z$.
Hence $G=A\rtimes \Z$. Now let $K$ be an open compact subgroup of $A$. Since it is open and $\Z^2$ is dense in $A$, $H_K=\Z^2\cap K$ is dense in $K$. Since $G$ is not discrete, $H_K$ must be non-trival. Let $u\in H_K\setminus \{1_G\}$. Since the conjugation by the generator $t$ in $\Z$ is continuous, we have that $t^{-1}ut$ must also be contained in a compact subgroup of $A$. But since $A$ is abelian, this implies that $t^{-1}ut$ and $u$ are contained in a compact subgroup of $A$. But since these are linearly independent  vectors in $\Q^2$, they generate a finite index subgroup of $\Z^2$. Hence since $\Z^2$ is dense in $A$, this would imply that $A$ is compact, again not compatible with $G$ being non-unimodular. So we are done.
\end{proof}

\section{Edge boundaries}\label{sec:edge}
In this paper, we chose to focus on the exterior boundary. The important aspect of this choice is that this boundary is a set of vertices, not of edges. The main reason for this choice comes from the fact that it is well suited for Cayley graphs of locally compact compactly generated groups: indeed the measure of the exterior boundary is well-behaved despite the fact that the graph may have infinite degree. Moreover, as seen in \cref{lem:quantitativeVT}, there is a nice connection between F\o lner sets in the group, and F\o lner sets in the graph on which $G$ acts transtively. 

Say that a graph $\Gamma$ is {\it edge-amenable} if there exists a sequence of finite subset $F_n$ such that $|\partial_e F_n|/|F_n|$ tends to zero, where the edge-boundary $\partial_e F_n$ is the set of edges joining a vertex of $F_n$ to a vertex of its complement. 
Since amenability and edge-amenability are obviously equivalent for bounded degree graphs,  this discussion is only relevant for the statements involving locally finite graphs with unbounded degree.  Let us focus our discussion here on \cref{thm:geo.amen<->action.simple}. The first statement remains true for the edge boundary: one implication is immediate by the previous discussion, and the other one follows from the simple observation that the size of the edge boundary is always larger than than that of the exterior boundary. 

However, the second statement of  \cref{thm:geo.amen<->action.simple} does not have an obvious analogue, motivating the following question. 
\begin{que}
Characterise those locally compact groups $G$ admitting a contingently $1$-Lipschitz proper action on a locally finite edge-amenable graph. 
\end{que}

\iffalse Note that there is a way to measure the edge boundary in a Cayley graph of a locally compact group compactly generated group, identifying the set of edges with $G\times S$, and equipping the latter with the square of the Haar measure. The issue with this approach is that when the Haar measure is multiplied by a scalar $t$, the size of the boundary changes by a factor of $t^2$, while the size of the set only changes by a factor of $t$. Hence the ratio of these two quantities now depends on a specific choice of Haar measure.
\fi

\section*{Ackowledgements}
We thank Itai Benjamini for an inspiring discussion, Russell Lyons for help with the references and corrections to an earlier draft, and an anonymous referee for a number of helpful suggestions, including to consider edge boundaries as we did in \cref{sec:edge}.


\begin{thebibliography}{10}
\bibitem{ab} H. Abels. Specker-Kompaktifizierungen von lokal kompakten topologischen
Gruppen. \textit{Math. Z.}, \textbf{135} (1973/74), 325--361.
\bibitem{ablrsv}
G. N. Arzhantseva, J. Burillo, M. Lustig, L. Reeves, H. Short and E. Ventura. Uniform non-amenability. \textit{Adv. Math.} \textbf{197}(2) (2005), 499--522.
\bibitem{BHV} B.~Bekka, P.~de~la Harpe, and A.~Valette.
\newblock {\em Kazhdan's property ({T})}, volume~11 of {\em New Mathematical
  Monographs}.
\newblock Cambridge University Press, Cambridge, 2008.
\bibitem{blps}
I. Benjamini, R. Lyons, Y. Peres and O. Schramm. Group-invariant percolation on graphs. \textit{Geom. Funct. Anal.} \textbf{9} (1999), 29--66.

\bibitem{CH} Y. Cornulier, P. de la Harpe. \textit{Metric geometry of locally compact groups}. EMS Tracts in Mathematics \textbf{25}, European Math. Society (2016).
\bibitem{drutu-kapovich}
 C. Dru\c{t}u and M. Kapovich. \textit{Geometric Group Theory}. With an appendix by Bogdan Nica. American Mathematical Society Colloquium Publications, 63. American
Mathematical Society, Providence, RI, 2018. xx+819 pp.
\bibitem{emerson-greenleaf}
W. R. Emerson and F. P. Greenleaf. Covering properties and F\o lner conditions for locally compact groups. \textit{Math. Z.} \textbf{102} (1967), 370--384.
\bibitem{hew-ross}
E. Hewitt and K. A. Ross. \textit{Abstract Harmonic Analysis I (2nd ed.)}, Springer-Verlag, Berlin (1979).
\bibitem{hutchcroft-tointon}
T. Hutchcroft and M. Tointon. Non-triviality of the phase transition for percolation on finite transitive graphs. 	arXiv:2104.05607
\bibitem{KM} B. Kr\"on, R. G. M\"oller. Analogues of Cayley graphs for topological groups.  \textit{Math. Z.} \textbf{258} (2008), 637--675.

\bibitem{ly-per}
R. Lyons and Y. Peres. \textit{Probability on Trees and Networks}, Cambridge Series in Statistical and Probabilistic Mathematics \textbf{42}, Cambridge University Press (2016).
\bibitem{meier} J. Meier. \textit{Groups, Graphs and Trees
An Introduction to the Geometry of Infinite Groups}, London Mathematical Society Student Texts \textbf{73}. Cambridge University Press (2008).  
\bibitem{osin.weak.amen}
D. V. Osin. Weakly amenable groups. \textit{Computational and statistical group theory (Las Vegas, NV/Hoboken, NJ, 2001)}, 105--113. Contemp. Math., \textbf{298}, Amer. Math. Soc., Providence, RI (2002). 
\bibitem{sc-woess}
L. Saloff-Coste and W. Woess. Transition operators on co-compact $G$-spaces. \textit{Rev. Mat. Iberoam.} \textbf{22}(3) (2006), 747--799.
\bibitem{salvatori}
M. Salvatori. On the norms of group-invariant transition operators on graphs. \textit{J. Theor Probab} \textbf{5} (1992), 563--576.  
\bibitem{soardi-woess}
P. M. Soardi and W. Woess. Amenability, unimodularity, and the spectral radius of random walks on infinite graphs, \textit{Math. Z.} \textbf{205} (1990), 471--486.
\bibitem{tessera.doubling}
R. Tessera. Volume of spheres in doubling metric measured spaces and in groups of polynomial growth. \textit{Bull. Soc. Math. France} \textbf{135}(1) (2007), 47--64.
\bibitem{tessera.sobolev}
R. Tessera. Large scale Sobolev inequalities on metric measure spaces and applications.  \textit{Rev. Mat. Iberoam.} \textbf{24}(3) (2008), 825--864.
\bibitem{tt.trof}
R. Tessera and M. C. H. Tointon. A finitary structure theorem for vertex-transitive graphs of polynomial growth, \textit{Combinatorica} \textbf{41} (2021), 263--298.
\bibitem{trof}
V. I. Trofimov. Graphs with polynomial growth, \textit{Math. USSR-Sb.} \textbf{51} (1985) 405--417.
\bibitem{woess} W. Woess. Topological groups and infinite graphs,  \textit{Discrete Math.} \textbf{95} (1991), 373--384.
\end{thebibliography}
\end{document}